\documentclass[11pt,dvips,twoside,letterpaper]{article}

\usepackage{pslatex}
\usepackage{fancyhdr}
\usepackage{graphicx}
\usepackage{geometry}
\usepackage[latin1]{inputenc}
\usepackage{mathrsfs}
\usepackage{amsmath}
\usepackage{amssymb}
\usepackage{amsmath}
\usepackage{amsfonts}
\usepackage{graphicx}
\usepackage{amsmath}
\usepackage[english]{babel}
\RequirePackage[latin1]{inputenc}
 \RequirePackage[T1]{fontenc}

\def\figurename{Figure} 
\makeatletter
\renewcommand{\fnum@figure}[1]{\figurename~\thefigure.}
\makeatother

\def\tablename{Table} 
\makeatletter
\renewcommand{\fnum@table}[1]{\tablename~\thetable.}
\makeatother

\usepackage{amsmath}
\usepackage{amssymb}
\usepackage{amsfonts}
\usepackage{amsthm,amscd}

\newtheorem{theorem}{Theorem}[section]
\newtheorem{lemma}[theorem]{Lemma}

\theoremstyle{definition}
\newtheorem{definition}[theorem]{Definition}

\theoremstyle{remark}
\newtheorem{remark}[theorem]{Remark}

\numberwithin{equation}{section}

\begin{document}

\title{Mean-field backward stochastic differential equations on Markov chains\thanks{The work of Wen Lu is supported
partially by the National Natural Science Foundation of China
(61273128) and a Project of Shandong  Province Higher Educational
Science and Technology Program (J13LI06). The work of Yong Ren is
supported by the National Natural Science Foundation of
 China (10901003),
 the Distinguished Young Scholars of Anhui Province (1108085J08),
 the Key
 Project of Chinese Ministry of Education (211077) and the Anhui Provincial Natural
 Science Foundation (10040606Q30).  }}

\author{Wen Lu$^1$\footnote{e-mail: llcxw@163.com}\;\  \ \ Yong
Ren$^2$\footnote{Corresponding author. e-mail: brightry@hotmail.com
and renyong@126.com}
\\
 \small  1. School of Mathematics and Informational Science, \\ \small Yantai University, Yantai 264005, China   \\
 \small 2. Department of Mathematics, \\
 \small Anhui Normal University, Wuhu
241000, China}
\date{}
\maketitle
\begin{abstract}
 In this paper, we deal with a class of mean-field backward stochastic differential equations (BSDEs) related to finite state,
 continuous time Markov chains. We obtain the existence and uniqueness theorem and a comparison theorem
for  solutions of one-dimensional mean-field BSDEs under  Lipschitz
condition.
\end{abstract}

\vspace{.08in} \noindent \textbf{Keywords}   Mean-field backward
stochastic differential equations;   Markov chain; comparison
theorem

\vspace{.08in} \noindent \textbf{Mathematics Subject Classification}
 60H20; 60H10
\section{Introduction}

The general (nonlinear) backward stochastic differential equations
(BSDE in short) were firstly introduced by Pardoux and Peng
\cite{PP} in 1990. Since then,  BSDEs have been studied with great
interest, and they have gradually become an important mathematical
tool in many fields such as financial mathematics, stochastic games
and optimal control, etc, see for example, Peng \cite{Peng},
Hamad\`{e}ne and Lepeltier \cite{HL95} and El Karoui et al.
\cite{EIK}.

\emph{McKean-Vlasov} stochastic differential equation of the form
\begin{eqnarray} \label{bsde:1}
  dX(t)=b(X(t),\mu(t))dt+dW(t),\quad t\in[0, T],\quad X(0)=x,
\end{eqnarray}
where
$$b(X(t),\mu(t))=\int_{\Omega}b(X(t,\omega), X(t;\omega'))P(d\omega')=E[b(\xi, X(t)]|_{\xi=X(t)},$$
$b: R^m\times R^m\rightarrow R$ being a (locally) bounded Borel
measurable function and $\mu(t; \cdot)$ being the probability
distribution of the unknown process $X(t)$, was suggested by Kac
\cite{Kac}  as a stochastic toy model for the Vlasov kinetic
equation of plasma and the study of which was initiated by Mckean
\cite{McKean}. Since then, many authors made contributions on
McKean-Vlasov type SDEs and applications, see for example, Ahmed
\cite{Ahmed}, Ahmed and Ding \cite{AhmedDing}, Borkar and Kumar
\cite{Borkar}, Chan \cite{Chan}, Crisan and Xiong
\cite{CrisanXiong}, Kotelenez \cite{Kotelenez}, Kotelenez and Kurtz
\cite{KotelenezKurtz}, and so on.

Mathematical mean-field approaches have been used in many fields,
not only in physics and Chemistry, but also recently in economics,
finance and game theory, see for example, Lasry and Lions
\cite{LasryLions}, they have studied mean-field limits for problems
in economics and finance, and also for the theory of stochastic
differential games.

Inspired by Lasry and Lions \cite{LasryLions}, Buckdahn et al.
\cite{Buckdahn1} introduced a new kind of BSDEs-mean-field BSDEs.
Furthermore, Buckdahn et al. \cite{Buckdahn2} deepened the
investigation of mean-field BSDEs in a rather general setting, they
gave the existence and uniqueness of solutions for mean-field BSDEs
with Lipschitz condition on coefficients, they also established the
comparison principle for these mean-field BSDEs. On the other hand,
since the works \cite{Buckdahn1} and \cite{Buckdahn2} on the
mean-field BSDEs, there are some efforts devote to its
generalization,  Xu \cite{Xu} obtained the existence and uniqueness
of solutions for mean-field backward doubly stochastic differential
equations; Li and Luo \cite{LiLuo} studied reflected BSDEs of
mean-field type, they proved the existence and the uniqueness for
reflected mean-field BSDEs; Li \cite{LiJuan} studied reflected
mean-filed BSDEs in a purely probabilistic method, and gave a
probabilistic interpretation of the nonlinear and nonlocal PDEs with
the obstacles.

However, most previous contributions to BSDEs and mean-field BSDEs
have been obtained in the framework of continuous time diffusion.
Recently,  Cohen and Elliott \cite{Cohen2008} introduced a new kind
of BSDEs of the form, for $t\in[0, T]$
\begin{eqnarray}\label{bsdem:1}
 Y_t=\xi+\int_t^Tf(s,Y_{s-},Z_s)ds-\int_t^TZ_sdM_s,
\end{eqnarray}
where  $M_t$ is a martingale related to a finite state continuous
time Markov chain (the details of $M_t$ will be given in Section 2).
In Cohen and Elliott \cite{Cohen2008}, the authors proved the
existence and uniqueness of solutions for those equations under
Lipschitz condition. Furthermore, Cohen and Elliott
\cite{Cohen2010a} gave a scalar and vector comparisons for solutions
of the BSDEs on Markov chains. Furthermore, they discussed arbitrage
and risk measure in scalar case.

 Very recently, Cohen and Elliott \cite{Cohen2010b}  established the
existence and uniqueness as well as comparison theorem for BSDEs in
general spaces.   In Cohen et al. \cite{CohenElliottPearce2010},
they established a general comparison theorem for BSDEs based on
arbitrary martingales and gave its applications to the theory of
nonlinear expectations.

Motivated by above works, the present paper deal with a class of
Mean-field BSDEs on Markov Chains of the form
\begin{eqnarray}\label{mfsdem:1}
 Y_t=\xi+\int_t^TE'[f(s,Y_{s-}',Z_{s}',Y_{s-},Z_{s})]ds-\int_t^TZ_{s}dM_s,
\end{eqnarray}
To the best of our knowledge, so far little is known about this new
kind of BSDEs. Our aim is to find a pair of adapted processes
$(Y,Z)$ in an appropriate space such that  \eqref{mfsdem:1}  hold.
We also present a comparison theorem for the solutions of BSDEs
\eqref{mfsdem:1}. We see that our BSDE \eqref{mfsdem:1} includes
BSDE \eqref{bsdem:1} as a special case.

The paper is organized as follows. In Section 2, we introduce some
preliminaries. Section 3 is devoted to the proof of the existence
and uniqueness of the solutions to Mean-field BSDEs on Markov chains
. In Section 4, we give a comparison theorem for the solutions of
Mean-field BSDEs.

\section{Preliminaries}

 Let $T>0$ be fixed throughout this paper. Let
$X=\{X_t, t\in[0, T]\}$ be a continuous time finite state Markov
chain. The states of this process can be identified with the unite
vector in $ R^N$, where $N$ is the number of states of the chain.

Let $(\Omega, \mathcal {F}, P)$ be a complete
 probability space. We denote by $\mathbb{F}=\{\mathcal {F}_t, 0\leq t\leq T\}$
the natural filtration generated by $X=\{X_t, t\in[0, T]\}$ and
augmented by all $P$-null sets, i.e.,
$$\mathcal {F}_t=\sigma\{X_u, 0\leq u\leq t\}\vee \mathscr{N}_P,$$
where $\mathscr{N}_P$ is the set of all $P$-null subsets.

Let $A_t$ be the rate matrix for the chain $X$ at time $t$, then
this chain has the representation
\begin{eqnarray*}\label{equation:1}
X_t=X_0+\int_0^tA_uX_{u-}du+M_t,
\end{eqnarray*}
where $M_t$ is a martingale related to the chain $X=\{X_t, t\in[0,
T]\}$. The optional quadratic  variation of $M_t$ is given by the
matrix process
\begin{eqnarray*}
[M,M]_t=\sum_{0<u\leq t}\Delta M_u\Delta M_u^*
\end{eqnarray*}
and
\begin{eqnarray*}
\langle M,M\rangle_t=\int_{]0, \; t]}[ {\rm diag} (A_uX_{u-})-{\rm
diag} (X_{u-})A_u^*-A_u{\rm diag} (X_{u-})]du,
\end{eqnarray*}
where $[\cdot]^*$ denotes matrix/vector transposition.

Let $\Phi_t$ be the nonnegative definite matrix

$$\Phi_t:={\rm diag}(A_tX_{t-})-{\rm
diag}(X_{t-})A_t^*-A_t{\rm diag}(X_{t-})$$ and
$$\|Z\|_{X_{t-}}:=\sqrt{{\rm Tr}(Z\Phi_tZ^*)}.$$
Then $\|\cdot\|_{X_{t-}}$ defines a (stochastic) seminorm, with the
property that

$${\rm Tr}(Z_td\langle M, M\rangle_tZ_t^*)=\|Z\|_{X_{t-}}^2dt.$$

Now, we provide some spaces and notations used in the sequel.
\begin{itemize}
\item  $L^{p}(\Omega, \mathcal{F}_T, P):=\{\xi:$   real
valued $\mathcal{F}_T$-measurable  random variable
$E|\xi|^p<+\infty, p\geq 1\}$;

\item  $L^{0}(\Omega, \mathcal{F}, P;  R^{n}):=\{\xi:$
$R^{n}$-valued $\mathcal{F}$-measurable random variable$\}$;

\item  $S^{2}_{\mathbb{F}}(R):=\{Y: \Omega\times [0,T]\rightarrow  R$
c\`{a}dl\`{a}g and $\mathbb{F}$-adapted, $E\big[\sup_{t\in[0, T]}
|Y_t|^2\big] <+\infty\}$;

\item  $H^{2}_{X,\mathbb{F}}( R^{N}):=\{Z:
\Omega\times [0,T]\rightarrow  R^{N}$, left continuous and
predictable, $E\int_0^T \|Z_t\|^2_{X_{t-}}dt<+\infty\}$.
\end{itemize}

 Let  $(\bar{\Omega}, \bar{\mathcal{F}},
\bar{P})=(\Omega\times\Omega, \mathcal{F}\otimes\mathcal{F},
P\otimes P )$ be the (non-completed) product of  $(\Omega,
\mathcal{F}, P)$ with itself. We denote the filtration of this
product space by
$\bar{\mathbb{F}}=\{\bar{\mathcal{F}}_t=\mathcal{F}\otimes\mathcal{F}_t,
0\leq t\leq T\}$. A random variable $\xi\in L^{0}(\Omega,
\mathcal{F}, P;  R^{n})$ originally defined on $\Omega$ is extended
canonically to $\overline{\Omega}:
\xi'(\omega',\omega)=\xi(\omega'), (\omega',\omega)\in
\bar{\Omega}=\Omega\times\Omega$. For any $\theta\in
L^{1}(\bar{\Omega}, \bar{\mathcal{F}}, \bar{P})$ the variable
$\theta(\cdot  ,\omega):\Omega\rightarrow  R$  belongs to
$L^{1}(\Omega, \mathcal{F}, P)$, $P(d\omega)$-a.s.; we denote its
expectation by
$$ E'[\theta(\cdot,\omega)]=\int_{\Omega}\theta(\omega',\omega)P(d\omega').$$
Notice that $ E'[\theta]= E'[\theta(\cdot,\omega)]\in L^{1}(\Omega,
\mathcal{F}, P)$,  and
$$\bar{ E}[\theta]\Big(=\int_{\Omega}\theta d\bar{P}
=\int_{\Omega} E'[\theta(\cdot,\omega)]P(d\omega)\Big)= E[
E'[\theta]].$$

For convenience, we rewrite mean-field BSDEs \eqref{mfsdem:1} as
below:
\begin{eqnarray}\label{mfsdem:2}
 Y_t=\xi+\int_t^TE'[f(s,Y_{s-}',Z_{s}',Y_{s-},Z_{s})]ds-\int_t^TZ_{s}dM_s.
\end{eqnarray}
The coefficient of our mean-field BSDE is a function $f=f(\omega',
\omega, t, y', z', y, z):\bar{\Omega}\times[0,T]\times
 R\times  R^N\times  R\times
 R^N\rightarrow   R$ which is
$\bar{\mathbb{F}}$-progressively measurable, for all $(y',z',y,z)$.
We make the following assumptions:

(A1) There exists a constant $C\geq 0$ such that, $dt\times
\bar{P}$-a.s., $y_1,y_2,y_1',y_2'\in R$,$z_1,z_2,z_1',z_2'\in R^N$,
\begin{eqnarray*}
&&|f(\omega',\omega, t, y_1',z_1',y_1,z_1)-f(\omega',\omega, t,
y_2',z_2',y_2,z_2)|\nonumber\\&&\quad\leq C\Big(|y_1'-y_2'|
+\|z_1'-z_2'\|_{X_{t-}} +|y_1-y_2| +\|z_1-z_2\|_{X_{t-}} \Big);
\end{eqnarray*}

(A2) $ \bar{E} \int_{0}^T|f(t,0,0,0,0)|^2dt <+\infty$.
\begin{remark}  Since the integral in \eqref{mfsdem:2} is with respect to Lebesgue measure and our processes have
at most countably many jumps, in this case the equation is unchanged
whether the left limits are included or not.
\end{remark}
\begin{remark} We emphasize that, due to our notations, the driving coefficient
$f$ of \eqref{mfsdem:2} has to be interpreted as follows
\begin{eqnarray*}
 E'[f(s,Y_s',Z_s',Y_s,Z_s)](\omega)&=&E'[f(s,Y_s',Z_s',Y_s(\omega),Z_s(\omega))]
\\&=&\int_{\Omega}f(s,Y_s'(\omega'),Z_s'(\omega'),Y_s(\omega),Z_s(\omega))P(d\omega').
\end{eqnarray*}
\end{remark}
\begin{definition}
 A solution to mean-filed
BSDE  \eqref{mfsdem:2} is  a couple $(Y,Z)=(Y_t,Z_t)_{0\leq t\leq
T}$ satisfying   \eqref{mfsdem:2} such that $(Y,Z)\in
S^{2}_{\mathbb{F}}(R) \times H^{2}_{X,\mathbb{F}}( R^{N})$.
\end{definition}

\section{Existence and uniqueness of solutions}
In this section, we aim to derive the
 existence and uniqueness result for the solutions
of mean-field BSDEs on Markov chains.

Before stating our main theorem, we  recall an existence and
uniqueness result in Cohen and Elliott  \cite{Cohen2008}, or more
precisely, in Cohen and Elliott \cite{Cohen2010b}.
\begin{lemma}\label{lemma:1}
Given $\xi\in L^{2}(\Omega, \mathcal{F}_T, P)$. Suppose assumptions
{\rm(A1)} and {\rm(A2)} hold. Then  BSDE {\rm \eqref{bsdem:1}} has a
unique solution $(Y,Z)\in  S^{2}_{\mathbb{F}}(R) \times
H^{2}_{X,\mathbb{F}}( R^{N})$, and the solution is the  unique such
solution, up to indistinguishability for $Y$ and  equality $d\langle
M, M\rangle_t\times P$-a.s. for $Z$.
\end{lemma}
For the solutions of mean-field BSDE {\rm \eqref{mfsdem:2}}, we
first establish the following  unique result.
\begin{lemma}\label{lemma:2}
Given $\xi\in L^{2}(\Omega, \mathcal{F}_T, P)$. Suppose assumptions
{\rm(A1)} and {\rm(A2)} hold. Then mean-field BSDE {\rm
\eqref{mfsdem:2}} has at most one solution $(Y,Z)\in
S^{2}_{\mathbb{F}}(R) \times H^{2}_{X,\mathbb{F}}( R^{N})$.
\end{lemma}
\begin{proof}
Let $(Y^i, Z^i)\in S^{2}_{\mathbb{F}}(R) \times
H^{2}_{X,\mathbb{F}}( R^{N}), i=1,2$  be two solutions of mean-field
BSDE {\rm \eqref{mfsdem:2}}.  Define $ \hat{Y}=Y^1-Y^2,
\hat{Z}=Z^1-Z^2$ , we then have
\begin{eqnarray*}
\hat{Y}(t)=\int_t^TE'[\hat{f}(s)]ds-\int_t^T\hat{Z}_sdM_s,
\end{eqnarray*}
where $\hat{f}(s)=f(s, Y^{1\prime}_{s-} ,
Z^{1\prime}_{s},Y_{s-}^1,Z_{s}^1) -f(s, Y^{2\,\prime}_{s-} ,
Z^{2\,\prime}_{s},Y_{s-}^2,Z_{s}^2)$.

 Using the Stieltjes chain rule for
products, we get
\begin{eqnarray}\label{ineq:14}
|\hat{Y}_t|^2&=&|\hat{Y}_0|^2-2\int_0^t \hat{Y}_{s-}E'[\hat{f}(s)]ds
 +2\int_0^t\hat{Y}_{s-}\hat{Z}_sdM_s+\sum_{0<s\leq t} |\Delta
Y^1_s-\Delta Y^2_s |^2.
 \end{eqnarray}
Taking expectation on both sides of \eqref{ineq:14} and evaluating
at $t=T$, we obtain
\begin{eqnarray}\label{ineq:16}
 E|\hat{Y}_t|^2&=&2\int_t^T E[\hat{Y}_{s-}E'[\hat{f}(s)]]ds
 - E\sum_{t<s\leq T}|\Delta Y^1_s-\Delta Y^2_s|^2
\nonumber \\&=&2\int_t^T E[\hat{Y}_{s-} E'[\hat{f}(s)]] ds
 - E\sum_{t<s\leq T}|(Z^1_s- Z^2_s)\Delta M_s|^2
  \nonumber \\&=& 2\int_t^T E[\hat{Y}_{s-}E'[\hat{f}(s)]]ds
 -\int_t^T E \|\hat{Z}_s\|_{X_{s-}}^2ds.
 \end{eqnarray}
On the other hand,  by (A1) and Young's inequality $2ab\leq
\frac{1}{\rho}a^2+\rho b^2$, for any $\rho>0$, it hold
\begin{eqnarray*}
 && 2\int_t^T E[\hat{Y}_{s-}E'[\hat{f}(s)]]ds
\\&\leq& 2C\int_t^T E\Big[\hat{Y}_{s-}E'\big[|\hat{Y}'_{s-}|
+\|\hat{Z}'_{s}\|_{X_{s-}}+|\hat{Y}_{s-}|+\|\hat{Z}_{s}\|_{X_{s-}}\big]\Big]
ds
\\&\leq&  4C\int_t^T E |\hat{Y}_{s-}|^2 ds+2C\int_t^T \big[\rho E|\hat{Y}_{s-}|^2
+\frac{1}{\rho}E\|\hat{Z}_{s}\|^2_{X_{s-}}\big]ds.
  \end{eqnarray*}
Choosing $\rho=3C$, we obtain
\begin{eqnarray*}\label{ineq:17}
 2\int_t^T E[\hat{Y}_{s-}E'[\hat{f}(s)]]ds
  &\leq& (6C^2+4C)\int_t^T E |\hat{Y}_{s-}|^2ds
+\frac{2}{3}\int_t^TE\|\hat{Z}_{s}\|^2_{X_{s-}}ds.
  \end{eqnarray*}
This together with \eqref{ineq:16} implies
\begin{eqnarray*}\label{ineq:18}
 E|\hat{Y}_t|^2+\frac{1}{3}\int_t^TE\|\hat{Z}_{s}\|^2_{X_{s-}}ds
\leq(6C^2+4C)\int_t^T E |\hat{Y}_{s-}|^2ds.
 \end{eqnarray*}
An application of Gr\"{o}nwall's inequality gives
$$E|\hat{Y}_t|^2=0,\quad E\|\hat{Z}_t\|^2_{X_{t-}}=0,$$
i.e., $Y^1_{t}=Y^2_{t}$ and $Z^1_{t}=Z^2_{t}$ $P$-a.s. for each $t$.
The proof is complete.
\end{proof}

Next, let's consider a simplified version of mean-field BSDEs
\eqref{mfsdem:2} as follows
\begin{eqnarray}\label{mfe:6}
 Y_t =\xi+\int_t^T E'[f(s, Y_{s-}', Y_{s-} ,
Z_s )]ds-\int_t^TZ_s dM_s.
 \end{eqnarray}
 We have the following existence and uniqueness result.
\begin{lemma}\label{lemma:3}
Given $\xi\in L^{2}(\Omega, \mathcal{F}_T, P)$. Suppose assumptions
{\rm(A1)} and {\rm(A2)} hold. Then  mean-field BSDE {\rm
\eqref{mfe:6}} has a unique solution $(Y,Z)\in S^{2}_{\mathbb{F}}(R)
\times H^{2}_{X,\mathbb{F}}( R^{N}) $.
\end{lemma}
\begin{proof}
Let $Y_t^0=0$,  $t\in[0, T]$, we consider the following mean-field
BSDE:
\begin{eqnarray}\label{mfe:8}
 Y_t^{n+1}=\xi+\int_t^T E'[f(s, Y_{s-}^{n \,\prime}, Y_{s-}^n,
Z_s^{n+1})]ds-\int_t^TZ_s^{n+1}dM_s.
 \end{eqnarray}
According to Lemma \ref{lemma:1}, we can define recursively
$(Y^{n+1}, Z^{n+1})$  be the solution of BSDE \eqref{mfe:8}. For
$t\in[0, T]$, we have
\begin{eqnarray}\label{mfe:9}
 Y_t^{n+1}- Y_t^{n}&=&\int_t^T E'[f(s, Y_{s-}^{n \,\prime}, Y_{s-}^n,
Z_s^{n+1})-f(s, Y_{s-}^{n-1 \,\prime}, Y_{s-}^{n-1},
Z_s^{n})]ds-\int_t^T(Z_s^{n+1}-Z_s^{n})dM_s \nonumber
\\&=&Y_0^{n+1}- Y_0^{n}- \int_0^t E'[f(s, Y_{s-}^{n \,\prime},
Y_{s-}^n, Z_s^{n+1})-f(s, Y_{s-}^{n-1 \,\prime}, Y_{s-}^{n-1},
Z_s^{n})]ds.
 \nonumber
\\&&-\int_0^t(Z_s^{n+1}-Z_s^{n})dM_s
 \end{eqnarray}
Using the Stieltjes chain rule for products, we have
\begin{eqnarray*}\label{mfe:10}
 &&|Y_t^{n+1}- Y_t^{n} |^2
\\&=& |Y_0^{n+1}- Y_0^{n}|^2-2\int_0^t
(Y_{s-}^{n+1}-Y_{s-}^{n})E'[f(s, Y_{s-}^{n \,\prime}, Y_{s-}^n,
Z_s^{n+1})-f(s, Y_{s-}^{n-1 \,\prime}, Y_{s-}^{n-1}, Z_s^{n})]ds
\\&&
+2\int_0^t(Y_{s-}^{n+1}-Y_{s-}^{n})(Z_s^{n+1}-Z_s^{n})dM_s
+\sum_{0<s\leq t} |\Delta Y_s^{n+1}- \Delta Y_s^{n} |^2.
 \end{eqnarray*}
Taking expectation and evaluating at $t=T$, we obtain
\begin{eqnarray}\label{mfe:11}
 E|Y_t^{n+1}- Y_t^{n}|^2&=&2E\int_t^T
\big[(Y_{s-}^{n+1}-Y_{s-}^{n})E'[f(s, Y_{s-}^{n \,\prime}, Y_{s-}^n,
Z_s^{n+1})-f(s, Y_{s-}^{n-1 \,\prime}, Y_{s-}^{n-1},
Z_s^{n})]\big]ds
 \nonumber\\&&
   -\int_t^T E \|Z_s^{n+1}- Z_s^{n}\|_{X_{s-}}^2ds,
 \end{eqnarray}
By  (A1) and Young's inequality, for any $\rho>0$, we have
\begin{eqnarray}\label{mfe:12}
 &&2E\int_t^T [(Y_{s-}^{n+1}-Y_{s-}^{n})E'[f(s, Y_{s-}^{n \,\prime},
Y_{s-}^n, Z_s^{n+1})-f(s, Y_{s-}^{n-1 \,\prime}, Y_{s-}^{n-1},
Z_s^{n})]]ds
  \nonumber\\&\leq&
 2CE\int_t^T
\Big[(Y_{s-}^{n+1}-Y_{s-}^{n})E'[|Y_{s-}^{n \,\prime}-Y_{s-}^{n-1
\,\prime}|+|Y_{s-}^{n}-Y_{s-}^{n-1}|+\|Z_s^{n+1}-
Z_s^{n}\|_{X_{s-}}]\Big]ds
 \nonumber\\&\leq& \frac{3C}{\rho} \int_t^T
 E|Y_{s-}^{n+1}-Y_{s-}^{n}|^2ds +2\rho C\int_t^T
   E|Y_{s-}^{n}-Y_{s-}^{n-1}|^2ds + \rho C\int_t^T
 E\|Z_s^{n+1}- Z_s^{n}\|_{X_{s-}}^2 ds.
 \end{eqnarray}
Choosing $\rho=\frac{1}{2C}$, combining \eqref{mfe:11} and
\eqref{mfe:12}, we then have
\begin{eqnarray}\label{mfe:13}
  && E|Y_t^{n+1}- Y_t^{n}|^2+\frac{1}{2}\int_t^T E \|Z_s^{n+1}- Z_s^{n}\|_{X_{s-}}^2ds
 \nonumber\\&\leq&  c[\int_t^TE|Y_{s}^{n+1}-Y_{s}^{n}|^2ds+
\int_t^TE|Y_{s}^{n}-Y_{s}^{n-1}|^2ds],
\end{eqnarray}
where $c=\max\{6C^2, 1\}$. Let $u^n(t)=\int_t^T
E|Y_{s}^{n}-Y_{s}^{n-1}|^2ds$, it follows from  \eqref{mfe:13}
\begin{eqnarray*} \label{mfe:14}
-\frac{d u^{n+1}(t)}{dt}(t)-cu^{n+1}(t)\leq cu^n(t),\quad
u^{n+1}(T)=0.
\end{eqnarray*}
Integration gives
\begin{eqnarray*} \label{mfe:15}
 u^{n+1}(t) \leq c\int_t^Te^{c(s-t)}u^n(s)ds.
\end{eqnarray*}
Iterating above inequality, we obtain
\begin{eqnarray*} \label{mfe:16}
 u^{n+1}(0) \leq \frac{(ce^{c})^{n}}{n!} u^1(0).
\end{eqnarray*}
This implies that $\{Y^{n}\}$ is a Cauchy sequence in $
S^{2}_{\mathbb{F}}(R)$. Then by (\ref{mfe:13}), $\{Z^{n}\}$ is a
Cauchy sequence in $H^{2}_{X,\mathbb{F}}( R^{N})$.

Passing to the limit on both sides of \eqref{mfe:8}, by (A2) and the
dominated convergence theorem, it follows that
$$Y:=\lim_{n\rightarrow\infty}Y^n,\quad Z:=\lim_{n\rightarrow\infty}Z^n$$
solves  BSDE \eqref{mfe:6}. The uniqueness is a direct consequence
of Lemma \ref{lemma:2}. The proof is complete.
\end{proof}
The main result of this section is the following theorem.
\begin{theorem}\label{theorem:2}
Assume  that {\rm(A1)} and {\rm(A2)} hold true. Then for any given
terminal conditions $\xi\in L^{2}(\Omega, \mathcal{F}_T, P)$, the
mean-field BSDE {\rm \eqref{mfsdem:2}} has a unique solution
$(Y,Z)\in S^{2}_{\mathbb{F}}(R) \times H^{2}_{X,\mathbb{F}}( R^{N})
$.
\end{theorem}
\begin{proof} According to Lemma \ref{lemma:2}, all we need
to prove is the existence of solution for mean-field BSDE {\rm
\eqref{mfsdem:2}}.

 let  $Z_t^0=\textbf{0}$,
$t\in[0, T]$, in virtue of Lemma \ref{lemma:3}, we can define
recursively the pair of processes $(Y^{n+1}, Z^{n+1})$ be the unique
solution of  the following mean-field BSDE:
\begin{eqnarray}\label{mfe:18}
 Y_t^{n+1}=\xi+\int_t^T E'[f(s, Y_{s-}^{n+1 \,\prime},Z_{s}^{n \,\prime}, Y_{s-}^{n+1},
Z_s^{n+1})]ds-\int_t^TZ_s^{n+1}dM_s.
 \end{eqnarray}
Using the same procedure as above, we get
\begin{eqnarray*}\label{mfe:19}
&& E|Y_t^{n+1}- Y_t^{n}|^2
\\&=&2E\int_t^T
[(Y_{s-}^{n+1}-Y_{s-}^{n})E'[f(s, Y_{s-}^{n+1 \,\prime}, Z_{s}^{n
\,\prime}, Y_{s-}^{n+1}, Z_s^{n+1})-f(s, Y_{s-}^{n \,\prime},
Z_{s}^{n-1 \,\prime}, Y_{s-}^{n}, Z_s^{n})]]ds
\\&&
   -\int_t^T E \|Z_s^{n+1}- Z_s^{n}\|_{X_{s-}}^2ds
\\&\leq&2CE\int_t^T
\Big[(Y_{s-}^{n+1}-Y_{s-}^{n})E'[|Y_{s-}^{n+1 \,\prime}-Y_{s-}^{n
\,\prime}|+|Y_{s-}^{n+1}-Y_{s-}^{n}| +\|Z_{s-}^{n
\,\prime}-Z_{s-}^{n-1
\,\prime}\|_{X_{s-}}\\&&\qquad+\|Z_{s-}^{n+1}-Z_{s-}^{n}\|_{X_{s-}}]\Big]ds
   -\int_t^T E \|Z_s^{n+1}- Z_s^{n}\|_{X_{s-}}^2ds.
 \end{eqnarray*}
With the help of (A1) and Young's inequality, for any $\rho>0$, we
have
\begin{eqnarray*}\label{mfe:20}
&& E|Y_t^{n+1}- Y_t^{n}|^2
 \\&\leq&2CE\int_t^T
[(Y_{s-}^{n+1}-Y_{s-}^{n})E'[|Y_{s-}^{n+1 \,\prime}-Y_{s-}^{n
\,\prime}|+|Y_{s-}^{n+1}-Y_{s-}^{n}|+\|Z_{s-}^{n
\,\prime}-Z_{s-}^{n-1
\,\prime}\|_{X_{s-}}\\&&\qquad+\|Z_{s-}^{n+1}-Z_{s-}^{n}\|_{X_{s-}}]ds
   -\int_t^T E \|Z_s^{n+1}- Z_s^{n}\|_{X_{s-}}^2ds
\\&\leq& (4C+\frac{2C}{\rho})\int_t^T
E\big[|Y_{s-}^{n+1}-Y_{s-}^{n}|^2ds  +\rho C \int_t^TE
\|Z_{s-}^{n}-Z_{s-}^{n-1}\|_{X_{s-}}^2\big]ds
\\&&
+ (\rho C-1)
 \int_t^T E \|Z_s^{n+1}- Z_s^{n}\|_{X_{s-}}^2ds.
 \end{eqnarray*}
Define $k=4C+\frac{2C}{\rho}$,  by the backward Gr\"{o}nwall's
inequality, we obtain
\begin{eqnarray}\label{mfe:21}
&& E |Y_t^{n+1}- Y_t^{n} |^2 \nonumber \\&\leq& \rho C \int_t^TE
\|Z_{s}^{n}-Z_{s}^{n-1}\|_{X_{s-}}^2 ds + (\rho C-1)
 \int_t^T E \|Z_s^{n+1}- Z_s^{n}\|_{X_{s-}}^2ds\nonumber
\\&&
+ke^{-kt}\int_t^Te^{-ks}\big[\int_s^T\rho C E
\|Z_{u}^{n}-Z_{u}^{n-1}\|_{X_{u-}}^2 du + (\rho C-1)
 \int_s^T E \|Z_u^{n+1}- Z_u^{n}\|_{X_{u-}}^2du\big]ds.
\nonumber
\\
 \end{eqnarray}
Choosing $\rho=\frac{1}{3C}$, we get
\begin{eqnarray*}\label{mfe:22}
&& \int_t^TE \|Z_{s}^{n+1}-Z_{s}^{n}\|_{X_{s-}}^2 ds
+ke^{-kt}\int_t^Te^{ks} \int_s^T E
\|Z_{u}^{n+1}-Z_{u}^{n}\|_{X_{u-}}^2 du ds
\\&\leq&
 \frac{1}{2}
 \Big[\int_t^T E \|Z_s^{n}- Z_s^{n-1}\|_{X_{s-}}^2ds
+ke^{-kt}\int_t^Te^{ks} \int_s^T E
\|Z_{u}^{n}-Z_{u}^{n-1}\|_{X_{u-}}^2 du ds\Big].
 \end{eqnarray*}
Iterating above inequality implies that $\{Z^n\}$ is a Cauchy
sequence in $ H^{2}_{X,\mathbb{F}}( R^{N})$ under the  equivalent
norm.

By \eqref{mfe:21}, we know that  $\{Y^n\}$ is a Cauchy sequence in $
H^{2}_{\mathbb{F}}(R)$. We denote their limits by $Y$ and $Z$
respectively. By (A2) and  the dominated convergence theorem,   for
any $t\in[0, T]$, we have
\begin{eqnarray*}
E\int_t^T |E'[f(s, Y_{s-}^{n+1 \,\prime},Z_{s}^{n \,\prime},
Y_{s-}^{n+1}, Z_s^{n+1})-f(s, Y_{s-}^{\,\prime},Z_{s}^{\,\prime},
Y_{s-}, Z_s)]|ds\rightarrow 0,\quad n \rightarrow \infty.
 \end{eqnarray*}
We now pass to the limit on both sides of \eqref{mfe:18}, it follows
that  $(Y, Z)$ is the unique solution of mean-filed BSDE
\eqref{mfsdem:2}.
\end{proof}

\section{A comparison theorem}

In this section, we discuss a comparison theorem for the solutions
of one-dimensional mean-field BSDEs on Markov chains.

 Let $(Y^1, Z^1)$ and
$(Y^2, Z^2)$ be respectively the solutions for the following two
mean-field BSDEs
\begin{eqnarray} \label{absdem:22}
 Y^i_t=\xi^i +\int_t^TE'[f_i(s, Y^{i\prime}_s,  Y^i_s, Z^{i\prime}_s, Z^i_s)]ds
-\int_t^TZ^i_sdM_s,
\end{eqnarray}
where $i=1,2$.

\begin{theorem}\label{theorem:6}
Assume that $f_1, f_2$ satisfy {\rm(A1)} and {\rm(A2)}, $\xi^1
,\xi^2 \in L^{2}(\Omega, \mathcal{F}_T, P)$. Moreover, we suppose:

(i) $\xi^1 \geq\xi^2$,  $P$-a.s.;

(ii) for any $t\in[0, T]$,  $f_1(\omega',\omega, t,Y^{2\prime}_t,
Z^{2\prime}_t, Y^2_t, Z^2_t)\geq f_2(\omega',\omega,
t,Y^{2\prime}_t, Z^{2\prime}_t, Y^2_t, Z^2_t)$, $\bar{P}$-a.s.;
 then   $Y^1_t\geq Y^2_t$ for all $t\in[0, T]$
componentwise.

It is then rue that $Y^1 \geq Y^2$ on $[0,T]$, $P$-a.s.
\end{theorem}

\begin{proof}We omit the $\omega', \omega$ and $s$ for clarity.
By assumption (i), $(\xi^2-\xi^1)^+=0$, a.s.. Since for $t\in[0,
T]$, $(Y_t^2- Y^1_t)^+=\frac{1}{2}[|Y_t^2- Y^1_t|+(Y_t^2- Y^1_t)]$,
then by the Stieltjes chain rule for products, we have
\begin{eqnarray*}\label{mfe:10}
 &&((Y_t^2- Y_t^1)^+)^2
\\&=&-2\int_t^T(Y_s^2- Y_s^1)^+d(Y_s^2- Y_s^1)^+-\sum_{t< s\leq T}\Delta(Y_s^2- Y_s^1)^+
\Delta(Y_s^2- Y_s^1)^+
\\&=&-\int_t^T(Y_s^2- Y_s^1)^+d[|Y_s^2- Y^1_s|+(Y_s^2- Y^1_s)]-\sum_{t< s\leq T}\Delta(Y_s^2- Y_s^1)^+
\Delta(Y_s^2- Y_s^1)^+
\\&=&-\int_t^T(Y_s^2- Y_s^1)^+d|Y_s^2- Y^1_s| -\int_t^T(Y_s^2- Y_s^1)^+d(Y_s^2- Y^1_s)-\sum_{t< s\leq T}\Delta(Y_s^2- Y_s^1)^+
\Delta(Y_s^2- Y_s^1)^+
\\&=&-2\int_t^TI_{\{Y_s^2>Y_s^1\}}(Y_s^2- Y_s^1)d(Y_s^2- Y^1_s)
-\sum_{t<s\leq T}I_{\{Y_s^2>Y_s^1\}}\Delta(Y_s^2- Y_s^1)
\Delta(Y_s^2- Y_s^1)
\\&=&-2c\int_t^TI_{\{Y_s^2>Y_s^1\}}(Y_s^2- Y_s^1)d(Y_s^2- Y^1_s)
-\sum_{t<s\leq T}I_{\{Y_s^2>Y_s^1\}}|(Z_s^2- Z_s^1)\Delta M_s|^2.
 \end{eqnarray*}
For $t\in[0,T]$, by assumption (ii), (A1) and Young's inequality,
for any $\rho>0$, we have
\begin{eqnarray*}\label{mfe:11}
 &&E((Y_t^2- Y_t^1)^+)^2+E\int_t^TI_{\{Y_s^2>Y_s^1\}}\|(Z_s^2- Z_s^1)\|_{X_{s-}}^2ds
\\&=&2E\int_t^TI_{\{Y_s^2>Y_s^1\}}(Y_s^2- Y_s^1) E'[f_2(Y^{2\prime}_s,  Z^{2\prime}_s, Y^2_s, Z^2_s)
-f_1(Y^{1\prime}_s, Z^{1\prime}_s,  Y^1_s, Z^1_s)] ds
\\&\leq&2E\int_t^TI_{\{Y_s^2>Y_s^1\}}(Y_s^2- Y_s^1) E'[f_1(Y^{2\prime}_s,  Z^{2\prime}_s, Y^2_s, Z^2_s)
-f_1(Y^{2\prime}_s, Z^{2\prime}_s,  Y^2_s, Z^2_s)] ds
\\&\leq&2CE\int_t^TI_{\{Y_s^2>Y_s^1\}}(Y_s^2- Y_s^1)[|Y_s^2- Y_s^1|+\|(Z_s^2- Z_s^1)\|_{X_{s-}}+
E'|Y^{2\prime}_s - Y^{1\prime}_s|+E'\|(Z^{2\prime}_s-
Z^{1\prime}_s)\|_{X_{s-}}]ds
\\&\leq&2C\int_t^TE ((Y_s^2- Y_s^1)^+)^2ds
+2CE\int_t^TI_{\{Y_s^2>Y_s^1\}}(Y_s^2-
Y_s^1)E[I_{\{Y_s^2>Y_s^1\}}|Y_s^2- Y_s^1|]ds\\&&
+\frac{2C}{\rho}\int_t^TE ((Y_s^2- Y_s^1)^+)^2ds +2\rho
CE\int_t^TI_{\{Y_s^2>Y_s^1\}}\|(Z_s^2- Z_s^1)\|_{X_{s-}}^2ds
\\&\leq&(4C+\frac{2C}{\rho})\int_t^TE ((Y_s^2- Y_s^1)^+)^2ds +2\rho
CE\int_t^TI_{\{Y_s^2>Y_s^1\}}\|(Z_s^2- Z_s^1)\|_{X_{s-}}^2ds.
 \end{eqnarray*}
 Choosing $\rho=\frac{1}{2C}$, it follows from Gronwall's inequality
 that
 $E((Y_t^2- Y_t^1)^+)^2=0, t\in[0,T]$. It is then rue that $Y^1 \geq Y^2$ on $[0,T]$, $P$-a.s.
 The proof is complete.
\end{proof}

\begin{remark}
Compare to the comparison results in Cohen and Elliott
\cite{Cohen2010a}, our assumptions on coefficients $f_1$ and $f_2$
are natural. Moreover, we don't make restrictions on the two
solutions, hence it's easier to use.
\end{remark}


\label{lastpage-01}
\end{document}